\documentclass[12pt,twoside]{amsart}
\usepackage{amssymb}
\usepackage{graphicx}
\usepackage[margin=1in]{geometry}

\newcommand{\beas}{\begin{eqnarray*}}
\newcommand{\eeas}{\end{eqnarray*}}
\newcommand{\bea}{\begin{eqnarray}}
\newcommand{\eea}{\end{eqnarray}}
\newcommand{\beq}{\begin{equation}}
\newcommand{\eeq}{\end{equation}}
\newcommand{\ben}{\begin{enumerate}}
\newcommand{\een}{\end{enumerate}}

\newtheorem{theorem}{Theorem}

\theoremstyle{definition}
\newtheorem{remark}[theorem]{Remark}

\usepackage{latexsym}
\usepackage{color,hyperref}
\definecolor{darkblue}{rgb}{0,0,0.6}
\hypersetup{colorlinks,breaklinks,linkcolor=darkblue,urlcolor=darkblue,anchorcolor=darkblue,citecolor=darkblue}

\author[Richard P. Stanley]{Richard P. Stanley}
\address{Department of Mathematics, MIT, Cambridge, MA 02139}
\email{rstan@math.mit.edu}

\author[Fabrizio Zanello]{Fabrizio Zanello}
\address{Department of Mathematical Sciences, Michigan Tech, Houghton, MI 49931}
\email{zanello@mtu.edu}

\title[On the asymptotics of the number of $O$-sequences of given length]{A note on the asymptotics of the number of $O$-sequences of given length}

\begin{document}

\begin{abstract} 
We look at the number $L(n)$ of $O$-sequences of length $n$. Recall that an $O$-sequence can be defined algebraically as the Hilbert function of a standard graded $k$-algebra, or combinatorially as the $f$-vector of a multicomplex. The sequence $L(n)$ was first
investigated in a recent paper by commutative algebraists Enkosky and
Stone, inspired by Huneke. In this note, we significantly improve both
of their upper and lower bounds, by means of a very short
partition-theoretic argument. In particular, it turns out that, for
suitable positive constants $c_1$ and $c_2$ and all $n>
2$, $$e^{c_1\sqrt{n}}\le L(n)\le e^{c_2\sqrt{n}\log n}.$$ It remains
an open problem to determine an exact asymptotic estimate for $L(n)$. 
\end{abstract}


\maketitle


In this brief note, we investigate the number $L(n)$ of
$O$-sequences $h=(h_0,h_1,\dots,h_e)$ of length $n$. In other words,
$h$ is the Hilbert function of some standard graded (artinian)
$k$-algebra $A=\oplus_{i=0}^eA_i$, where $h_i=\dim_k A_i$ and
$\sum_{i=0}^eh_i=n$. Equivalently, in purely combinatorial terms, an $O$-sequence can also be viewed as the $f$-vector of a multicomplex.  We refer to \cite{St} for any standard facts or unexplained terminology.

Even though an explicit description of which integer vectors arise as
$O$-sequences has been around for nearly a century, thanks to a
celebrated result of Macaulay \cite{Ma}, many enumerative properties
of these important Hilbert functions remain obscure to this day (see
\cite{Li} for some progress). A very natural and intriguing problem in
this line of research is the study of $L(n)$, recently initiated in
\cite{ES}, where nontrivial upper and lower bounds were shown.  Interestingly, this problem was already posed back in 1994, though it appears that no progress was made at the time; see Roberts \cite{Ro}.  (We thank Adam Van Tuyl for providing this reference.)

Our present goal is to greatly simplify and improve both bounds given
in \cite{ES}, by means of the next theorem. Our lower bound (which is
asymptotically better than the one in \cite{ES} by roughly an exponent
of $\sqrt{2}$) boils down to a simple consideration; our more
significant improvement is in the upper bound, which has logarithmic
order $\sqrt{n}\log n$ rather than $n$. 

\begin{theorem}
For suitable constants $c_1,c_2>0$ and all $n>2$, we have
$$e^{c_1\sqrt{n}}\le L(n)\le e^{c_2\sqrt{n}\log n}.$$
\end{theorem}

\begin{proof}
The lower bound is trivial, since by Macaulay's theorem, any sequence $(1,h_1,\dots,h_e)$ where the $h_i$ are nonincreasing is an $O$-sequence. Thus, $L(n)\ge p(n-1)$, the number of integer partitions of $n-1$. It is well known \cite{HR} that $p(n)$ is asymptotic to $\frac{1}{4n\sqrt{3}}e^{\pi\sqrt{2n/3}}$ for $n$ large, whence our bound easily follows. (Up to polynomial factors, this asymptotically improves by an exponent of $\sqrt{2}$ the bound $L(n)\ge q(n)$ shown in \cite{ES}, where $q(n)$ denotes the number of distinct-part partitions of $n$, which required a substantial amount of work. Our bound is also sharper for each $n$, since $p(n-1)\ge q(n)$, with strict inequality for $n\ge 4$.)

As for the upper bound, consider the smallest index $j$ that satisfies $h_j\le j$. It immediately follows from Macaulay's theorem that $h_i\ge h_{i+1}$ for all $i\ge j$. Hence the final portion of $h$, $(h_j,\dots,h_e)$, coincides with an integer partition of $n - (h_0+\dots+h_{j-1})$, and therefore the number of possible choices for $(h_j,\dots,h_e)$ can be bounded from above by $p(n)$.

Now note that $j$ is smaller than $\sqrt{2n}$, since $h$ has length $n$. Further, the number of possible choices for the first portion of $h$, $(h_0,\dots,h_{j-1})$, is trivially bounded by $n^{\sqrt{2n}}=e^{\sqrt{2n} \log n}$. We conclude that the total number of $O$-sequences of length $n$ is at most
$$\sqrt{2n} \cdot p(n)\cdot  e^{\sqrt{2n} \log n} \ll e^{c \sqrt{n} \log n},$$
and the theorem follows.
\end{proof}

\begin{remark} While trying to determine a nice, explicit formula for
  $L(n)$ might be a hopeless task, it would be interesting to at least
  establish its precise asymptotic value, by further improving one (or
  both?) of our bounds. Also intriguingly it is reasonable to continue
  to expect a connection between this problem and estimates of
  partition-theoretic interest. In view of our theorem, the degrees
  where it is critical to control the behavior of $h$ lie roughly
  between $\sqrt{n}/\log n$ and  $j$, if this latter is larger. We
  only point out here that the $h_i$ can be nicely described in that range
  (in fact, as early as in degrees of order $n^{1/3}$); namely, we
  have $$h_i=\binom{i+1}{i}+\binom{i}{i-1}+\dots+\binom{i-t_i+1}{i-t_i}+\alpha_i,$$  
for suitable integers $t_i\ge 0$ and $0\le \alpha_i< i-t_i$ such that the $t_i$ are nonincreasing, and the $\alpha_i$ are nonincreasing throughout each range where the corresponding $t_i$ are constant.
\end{remark}

\section*{Acknowledgements} We wish to thank Adam Van Tuyl for informing us of reference \cite{Ro}. The key ideas of this paper were discussed during a visiting professorship of the second author in Fall 2017, for which he warmly thanks the first author and the MIT Math Department. The second author was partially supported by a Simons Foundation grant (\#274577).


\end{document}